\documentclass[11pt, reqno]{amsart}

\usepackage{amscd,amssymb,amsopn,amsmath,amsthm,mathrsfs,graphics,amsfonts,enumerate,verbatim,calc,
}
\usepackage{mathrsfs}
\usepackage[all]{xy}
\usepackage[normalem]{ulem}
\newcommand\redout{\bgroup\markoverwith
{\textcolor{red}{\rule[.5ex]{2pt}{0.4pt}}}\ULon}
\usepackage{xcolor}
\usepackage[colorlinks=true,linkcolor=blue,citecolor=blue]{hyperref}
\usepackage{color}

\usepackage[OT2,OT1]{fontenc}
\newcommand\cyr{%
\renewcommand\rmdefault{wncyr}%
\renewcommand\sfdefault{wncyss}%
\renewcommand\encodingdefault{OT2}%
\normalfont
\selectfont}
\DeclareTextFontCommand{\textcyr}{\cyr}

\usepackage{amssymb,amsmath}

\usepackage[top=1in, bottom=1in, right=1in, left=1in]{geometry}
\newtheorem{theorem}{Theorem}[section]
\newtheorem{lemma}[theorem]{Lemma}
\newtheorem{proposition}[theorem]{Proposition}

\newtheorem{corollary}[theorem]{Corollary}
\newtheorem{claim}[theorem]{Claim}

\newtheorem{conjecture}[theorem]{Conjecture}

\theoremstyle{definition}

\newtheorem{remark}[theorem]{Remark}

\theoremstyle{remark}

\newcommand{\Spec}{\operatorname{Spec}}

\newcommand{\Supp}{\operatorname{Supp}}

\newcommand{\Hom}{\operatorname{Hom}}

\newcommand{\m}{\frak{m}}

\newcommand{\n}{\frak{n}}

\begin{document}

\title[Vanishing and non-negativity of the first normal Hilbert coefficient]{Vanishing and non-negativity of the first normal Hilbert coefficient}

\author{Linquan Ma}
\address{Department of Mathematics, Purdue University, West Lafayette, IN 47907 USA}
\email{ma326@purdue.edu}

\author{Pham Hung Quy}
\address{Department of Mathematics, FPT University, Hanoi, Vietnam}
\email{quyph@fe.edu.vn}


\maketitle


\begin{center}
{\textit{Dedicated to Professor Ngo Viet Trung on the occasion of his 70th birthday}}
\end{center}

\begin{abstract}
Let $(R,\m)$ be a Noetherian local ring such that $\widehat{R}$ is reduced. We prove that, when $\widehat{R}$ is $S_2$, if there exists a parameter ideal $Q\subseteq R$ such that $\bar{e}_1(Q)=0$, then $R$ is regular and $\nu(\m/Q)\leq 1$. This leads to an affirmative answer to a problem raised by Goto-Hong-Mandal \cite{GHM}. We also give an alternative proof (in fact a strengthening) of their main result. In particular, we show that if $\widehat{R}$ is equidimensional, then $\bar{e}_1(Q)\geq 0$ for all parameter ideals $Q\subseteq R$, and in characteristic $p>0$, we actually have $e_1^*(Q)\geq 0$. Our proofs rely on the existence of big Cohen-Macaulay algebras.
\end{abstract}

\maketitle

\thispagestyle{empty}

\section{Introduction}
Let $(R,\m)$ be a Noetherian local ring of dimension $d$ such that $\widehat{R}$ is reduced and let $I\subseteq R$ be an $\m$-primary ideal. Then for $n\gg0$, $\ell(R/\overline{I^{n+1}})$ agrees with a polynomial in $n$ of degree $d$, and we have integers $\overline{e}_0(I),\dots, \overline{e}_d(I)$ such that
$$\ell(R/\overline{I^{n+1}})=\overline{e}_0(I)\binom{n+d}{d} - \overline{e}_1(I)\binom{n+d-1}{d-1} + \cdots + (-1)^d\overline{e}_d(I).$$
These integers $\overline{e}_i(I)$ are called the normal Hilbert coefficients of $I$.

It is well-known that $\overline{e}_0(I)$ is the Hilbert-Samuel multiplicity of $I$, which is always a positive integer. In this paper, we are interested in the first coefficient $\overline{e}_1(I)$. It was proved by Goto-Hong-Mandal \cite{GHM} that when $\widehat{R}$ is unmixed, $\overline{e}_1(I)\geq 0$ for all $\m$-primary ideals $I\subseteq R$ (which answers a question posed by Vasconcelos \cite{V}). They proposed a further problem in \cite[Section 3]{GHM} regarding the vanishing of $\overline{e}_1(I)$ and the regularity of the normalization of $R$. Since any $\m$-primary ideal $I$ is integral over a parameter ideal when the residue field is infinite, to study $\overline{e}_1(I)$ we may assume that $I=Q$ is a parameter ideal (i.e., it is generated by a system of parameters). In this paper, we prove the following main result which will lead to an affirmative answer to the question proposed in \cite{GHM}. This theorem is also a generalization of the main result of \cite{MTV}.

\begin{theorem}[Theorem \ref{theorem: Vanishing}]
Let $(R,\m)$ be a Noetherian local ring such that $\widehat{R}$ is reduced and $S_2$. If $\overline{e}_1(Q)=0$ for some parameter ideal $Q\subseteq R$, then $R$ is regular and $\nu(\m/Q)\leq 1$.
\end{theorem}

In \cite{GMV}, it was shown that when $R$ has characteristic $p>0$, for $n\gg0$, $\ell(R/(I^{n+1})^*)$ also agrees with a polynomial of degree $d$ and one can define the tight Hilbert coefficients $e_0^*(I),\dots, e_d^*(I)$ in a similar way (see Section 2 for more details). It is easy to see that $\overline{e}_1(I)\geq e_1^*(I)$. We strengthen the main result of \cite{GHM} in characteristic $p>0$ by showing that $e_1^*(Q)\geq 0$ for any parameter ideal $Q\subseteq R$ under mild assumptions.

\begin{theorem}[Corollary \ref{corollary: tightHilbertCoefficient}]
Let $(R,\m)$ be an excellent local ring of characteristic $p>0$ such that $\widehat{R}$ is reduced and equidimensional. Then we have $e^*_1(Q)\geq 0$ for all parameter ideals $Q\subseteq R$.
\end{theorem}

Our proofs of both theorems rely on the existence of big Cohen-Macaulay algebras. In fact, we show that the tight Hilbert coefficients $e_i^*(I)$ is a special case of what we call the BCM Hilbert coefficients $e_i^B(I)$ associated to a big Cohen-Macaulay algebra $B$, and the latter can be defined in arbitrary characteristic. In this context, we will show in Theorem \ref{theorem: Nonnegativity} that $\overline{e}_1(Q)\geq e_1^B(Q)\geq 0$ for all parameter ideals $Q\subseteq R$ when $B$ satisfies some mild assumptions. This recovers and extends the main result of \cite{GHM} in arbitrary characteristic.

Throughout this article, all rings are commutative with multiplicative identity 1. We will use $(R,\m)$ to denote a Noetherian local ring with unique maximal ideal $\m$. We refer the reader to \cite[Chapter 1-4]{BH} for basic notions such as Cohen-Macaulay rings, regular sequence, Euler characteristic, integral closure, and the Hilbert-Samuel multiplicity. We refer the reader to \cite[Section 07QS]{Sta} for the definition and basic properties of excellent rings. The paper is organized as follows. In Section 2 we collect the definitions and some basic results on big Cohen-Macaulay algebras and variants of Hilbert coefficients. In Section 3 we prove our main results and we propose some further questions.

\subsection*{Acknowledgement} This article was developed during many visits of the second author to Vietnam
Institute for Advanced Studies in Mathematics. He thanks sincerely the institute for their hospitality and valuable supports. We would like to thank Bernd Ulrich for valuable discussions, in particular for explaining to us some arguments in \cite{HU}. We would also like to thank the referee for his/her comments that lead to improvement of the paper. The first author was supported in part by NSF Grant DMS \#1901672, \#2302430, NSF FRG Grant \#1952366, and a fellowship from the Sloan Foundation. The second author is partially supported by a fund of Vietnam National Foundation for Science and Technology Development (NAFOSTED) under grant number 101.04-2023.08.

\section{Preliminaries}

Recall that an element $x$ in a ring $R$ is integral over an ideal $I\subseteq R$ if it satisfies an equation of the form $x^n+ a_1x^{n-1}+\cdots + a_{n-1}x + a_n=0$, where $a_k\in I^k$. The set of all elements integral over $I$ forms an ideal and is denoted by $\overline{I}$, called the integral closure of $I$. An ideal $I\subseteq R$ is called integrally closed if $I=\overline{I}$. It is well-known that an element $x\in R$ is integral over $I$ if and only if the image of $x$ in $R/\mathfrak{p}$ is integral over $I(R/\mathfrak{p})$ for all minimal primes $\mathfrak{p}$, see \cite[Proposition 1.1.5]{HS}.

Suppose $R$ is a Noetherian ring of prime characteristic $p>0$. The tight closure of an ideal $I\subseteq R$, introduced by Hochster--Huneke, is defined as follows:
$$I^* := \{x\in R \,\ | \,\ \text{there exists $c\in R-\cup_{\mathfrak{p}\in \text{Min}(R)}\mathfrak{p}$ such that $cx^{p^e}\in I^{[p^e]}$ for all $e\gg0$}\}.$$
An ideal $I\subseteq R$ is called tightly closed if $I=I^*$. In general, tight closure is always contained in the integral closure, that is, $I^*\subseteq\overline{I}$ (see \cite[Proposition on Page 58]{Ho07}). Similar to integral closure, an element $x\in R$ is in the tight closure of $I$ if and only if the image of $x$ in $R/\mathfrak{p}$ is in the tight closure of $I(R/\mathfrak{p})$ for all minimal primes $\mathfrak{p}$, see \cite[Theorem on page 49]{Ho07}.

Let $R$ be a Noetherian complete local domain and let $I\subseteq R$ be an ideal. The solid closure of $I$, denoted by $I^\bigstar$, consists of those element $x\in R$ such that there exists an $R$-algebra $S$ such that $\Hom_R(S, R)\neq 0$ and such that $x\in IS$. One can define solid closure of ideals in more general rings, see \cite[Definition 1.2]{Ho94}, but we will only need this notion for complete local domains. It was shown in \cite[Theorem 5.10]{Ho94} that solid closure is contained in the integral closure, i.e., $I^\bigstar\subseteq \overline{I}$. If $R$ has prime characteristic $p>0$, then solid closure agrees with tight closure $I^\bigstar = I^*$, see \cite[Theorem 8.6]{Ho94}.

\subsection{Big Cohen-Macaulay algebras}
Let $(R,\m)$ be a Noetherian local ring. An $R$-algebra $B$, not necessarily Noetherian, is called balanced big Cohen-Macaulay over $R$ if every system of parameters of $R$ is a regular sequence on $B$ and $\m B\neq B$. Balanced big Cohen-Macaulay algebras exist, in equal characteristic, this is due to Hochster-Huneke \cite{HH95}, and in mixed characteristic, this is proved by Andr\'{e} \cite{A18} (see also \cite{HM18,A20,B20}). In this article, we need to compare the closure operation induced by a balanced big Cohen-Macaulay algebra with integral closure. We begin with the following result. 

In what follows, when $R\to S$ is a (not necessarily injective) homomorphism of rings, $IS\cap R$ should be interpreted as the contraction of $IS$ to $R$. That is, those elements of $R$ whose image in $S$ are contained in $IS$. 

\begin{lemma}
\label{lemma: bigCM}
Let $(R,\m)$ be a Noetherian local ring. Then the following conditions are equivalent:
\begin{enumerate}
  \item $\widehat{R}$ is equidimensional.
  \item There exists a balanced big Cohen-Macaulay $R$-algebra $B$ such that
  \begin{equation}
  \label{eqn: equation dagger}
  I^B:= IB\cap R \subseteq \overline{I} \text{ for all $\m$-primary ideals $I\subseteq R$ }. \tag{$\dagger$}
  \end{equation}
  \item There exists a balanced big Cohen-Macaulay $R$-algebra $B$ such that $I^B\subseteq \overline{I}$ for all $I\subseteq R$.
\end{enumerate}
\end{lemma}
\begin{proof}
Since $(3)\Rightarrow(2)$ is obvious, we only need to show $(1)\Rightarrow(3)$ and $(2)\Rightarrow(1)$. Suppose $\widehat{R}$ is equidimensional and let $P_1,\dots,P_n$ be the minimal primes of $\widehat{R}$. Let $B_i$ be any balanced big Cohen-Macaulay algebra over $\widehat{R}/P_i$. Since $\widehat{R}$ is equidimensional, each system of parameters of $\widehat{R}$ is also a system of parameters of $\widehat{R}/P_i$ and thus $B_i$ is a balanced big Cohen-Macaulay algebra over $\widehat{R}$. It follows that $B:= \prod_{i=1}^{n}B_i$ is a balanced big Cohen-Macaulay algebra over $\widehat{R}$.
\begin{claim}
$(I\widehat{R})^B= IB\cap \widehat{R} \subseteq \overline{I\widehat{R}}$.
\end{claim}
\begin{proof}[Proof of Claim]
Since integral closure can be checked after modulo each minimal prime, it suffices to show that $(I\widehat{R})^B\cdot(\widehat{R}/P_i)\subseteq \overline{I(\widehat{R}/P_i)}$. It is easy to see (by our construction of $B$) that $$(I\widehat{R})^B\cdot(\widehat{R}/P_i) = (I(\widehat{R}/P_i))^{B_i}.$$
Since $B_i$ is a solid algebra over the complete local domain $\widehat{R}/P_i$ by \cite[Corollary 2.4]{Ho94}, we have
$$(I(\widehat{R}/P_i))^{B_i} \subseteq (I(\widehat{R}/P_i))^\bigstar\subseteq\overline{I(\widehat{R}/P_i)},$$
where the second inclusion follows from  \cite[Theorem 5.10]{Ho94}.
\end{proof}
By the claim above, we have
$$I^B\subseteq (I\widehat{R})^B \cap R \subseteq \overline{I\widehat{R}} \cap R=\overline{I},$$
where the last equality follows from \cite[Proposition 1.6.2]{HS}.

We next assume there exists a balanced big Cohen-Macaulay $R$-algebra $B$ that satisfies (\ref{eqn: equation dagger}). We first note that $\widehat{B}$ (the $\m$-adic completion of $B$) is still a balanced big Cohen-Macaulay algebra over $\widehat{R}$ by \cite[Corollary 8.5.3]{BH}. If $I$ is an $\frak m$-primary ideal, then we have $R/I \cong \widehat{R}/I\widehat{R} $ and $B/IB \cong \widehat{B}/I\widehat{B}$ (see \cite[Tag 05GG]{Sta}). It follows that $(I\widehat{R})^{\widehat{B}} = (I^B)\widehat{R} \subseteq \overline{I} \widehat{R} = \overline{I\widehat{R}}$ (where the last equality follows from \cite[Lemma 9.1.1]{HS}). Thus without loss of generality, we may replace $R$ by $\widehat{R}$ and $B$ by $\widehat{B}$ to assume $R$ is complete. Suppose $R$ is not equidimensional. Let $P_1, \ldots, P_n$ be all the minimal primes of $R$ such that $\dim (R/P_i) = \dim(R)$, and $Q_1, \ldots, Q_m$ be all the minimal primes of $R$ such that $\dim (R/Q_j) < d$. We pick $y \in Q_1 \cap \cdots \cap Q_m \setminus P_1\cup \cdots\cup P_n$. Then $y$ is a parameter element in $R$, and thus $y$ is a nonzerodivisor on $B$, since $B$ is balanced big Cohen-Macaulay. Since $y \cdot (P_1 \cap \cdots \cap P_n) \subseteq \sqrt{0}$, there exists $t$ such that $y^t \cdot (P_1 \cap \cdots \cap P_n)^t=0$. It follows that $(P_1 \cap \cdots \cap P_n)^tB = 0$. Hence $$(P_1 \cap \cdots \cap P_n)^t \subseteq \frak m^k B \cap R \subseteq \overline{\frak m^k}$$ for all $k$ by (\ref{eqn: equation dagger}). Thus $(P_1 \cap \cdots \cap P_n)^t\subseteq \cap_k \overline{\m^k} = \sqrt{0}$ by \cite[Exercise 5.14]{HS}, which is a contradiction. 
\end{proof}

\begin{remark}
In the proof of Lemma \ref{lemma: bigCM}, we have proved the fact that when $(R,\m)$ is a Noetherian complete local domain, then every balanced big Cohen-Macaulay algebra $B$ satisfies (\ref{eqn: equation dagger}). We suspect that when $(R,\m)$ is Noetherian, complete, reduced and equidimensional, then every balanced big Cohen-Macaulay algebra $B$ such that $\Supp(\widehat{B})=\Spec(R)$ satisfies (\ref{eqn: equation dagger}).
\end{remark}

\subsection{Hilbert coefficients}
Let $(R, \frak m)$ be a Noetherian local ring of dimension $d$ and let $I\subseteq R$ be an $\m$-primary ideal. Then for all $n \gg 0$ we have
$$\ell(R/I^{n+1}) = e_0(I) \binom{n+d}{d} - e_1(I) \binom{n+d-1}{d-1} + \cdots + (-1)^d e_d(I),$$
where $e_0(I), \cdots, e_d(I)$ are all integers, and are called the Hilbert coefficients of $I$.

Now suppose $R \oplus \overline{I}t \oplus \overline{I^2}t^2 \oplus \cdots $ is module-finite over the Rees algebra $R[It]$. For instance, by a famous result of Rees (see \cite[Corollary 9.2.1]{HS}), this is the case when $\widehat{R}$ is reduced. Then one can show that for all $n\gg0$, $\ell(R/\overline{I^{n+1}})$ agrees with a polynomial in $n$ and one can write
$$\ell(R/\overline{I^{n+1}}) = \overline{e}_0(I) \binom{n+d}{d} - \overline{e}_1(I) \binom{n+d-1}{d-1} + \cdots + (-1)^d \overline{e}_d(I),$$
where the integers $\overline{e}_0(Q), \cdots, \overline{e}_d(Q)$ are called the normal Hilbert coefficients. It is well-known that $e_0(I)=\overline{e}_0(I)$ agrees with the Hilbert-Samuel multiplicity $e(I, R)$ of $I$. 

We also recall the tight Hilbert coefficients studied in \cite{GMV}. Again, we suppose that $\widehat{R}$ is reduced and $R$ has characteristic $p>0$. Then we have 
$$\ell(R/(I^{n+1})^*) = e_0^*(I) \binom{n+d}{d} - e_1^*(I) \binom{n+d-1}{d-1} + \cdots + (-1)^d e_d^*(I),$$
for all $n\gg 0$, and the integers $e_0^*(I),\dots,e_d^*(I)$ are called the tight Hilbert coefficients, see \cite{GMV} for more details.

Now if $B$ is a balanced big Cohen-Macaulay $R$-algebra that satisfies (\ref{eqn: equation dagger}), then we know that $R\oplus I^Bt \oplus (I^2)^Bt^2 \oplus \cdots$ is an $R$-algebra that is also module-finite over $R[It]$: the fact that it is an $R$-algebra follows from the fact that $(I^a)^B(I^b)^B\subseteq (I^{a+b})^B$ for all $a, b$ (i.e., $\{(I^n)^B\}_n$ form a graded family of ideals), and that it is module-finite over $R[It]$ follows because by (\ref{eqn: equation dagger}), it is an $R[It]$-submodule of $R \oplus \overline{I}t \oplus \overline{I^2}t^2 \oplus \cdots $, and the latter is module-finite over $R[It]$ (note that $R[It]$ is Noetherian). Based on the discussion above, one can show that for all $n\gg0$, $\ell(R/(I^{n+1})^B)$ also agrees with a polynomial in $n$, and we write
$$\ell(R/(I^{n+1})^B) = e_0^B(I) \binom{n+d}{d} - e_1^B(I) \binom{n+d-1}{d-1} + \cdots + (-1)^d e_d^B(I),$$
for all $n\gg 0$ (see \cite{HM} for more general results). We call the integers $e_0^B(I),\dots, e_d^B(I)$ the BCM Hilbert coefficients with respect to $B$. It is easy to see that $e_0^B(I)=e(I, R)$ is still the Hilbert-Samuel multiplicity of $I$, and that we always have $\overline{e}_1(I)\geq e_1^B(I)\geq e_1(I)$ by comparing the coefficients of $n^{d-1}$ and noting that $I^n\subseteq(I^n)^B\subseteq \overline{I^n}$ for all $n$ by (\ref{eqn: equation dagger}).

\begin{remark}
\label{remark: tightvsBCM}
We point out that when $(R,\m)$ is excellent and $\widehat{R}$ is reduced and equidimensional of characteristic $p>0$, the tight Hilbert coefficient is a particular case of BCM Hilbert coefficient. This follows from the fact that under these assumptions, there exists a balanced big Cohen-Macaulay algebra $B$ such that $I^*=I^B$ for all $I\subseteq R$ (and any such $B$ will satisfy (\ref{eqn: equation dagger}), since tight closure is contained in the integral closure \cite[Theorem 1.3]{H98}). When $R$ is a complete local domain this is proved in \cite[Theorem on page 250]{Ho07}. In general, one can take such a $B_i$ for each complete local domain $\widehat{R}/P_i$, where $P_i$ is a minimal prime of $\widehat{R}$, and let $B=\prod B_i$. Since $R$ is excellent, $I^*\widehat{R}=(I\widehat{R})^*$ (see \cite[Proposition 1.5]{H98}) and as tight closure can be checked after modulo each minimal prime, it follows that $I^*\widehat{R}=(I\widehat{R})^B$ and thus $I^*=I^B$.
\end{remark}

Throughout the rest of this article, we will be mainly working with parameter ideals, i.e., ideals generated by a system of parameters. As we mentioned in the introduction, this will not affect the study of $\overline{e}_1(I)$, since we can often enlarge the residue field and replace $I$ by its minimal reduction.

\section{The main results}

In this section we prove our main results that $e_1^B(Q)$ (and hence $\overline{e}_1(Q)$) is always nonnegative for a parameter ideal $Q$, and that $\overline{e}_1(Q)=0$ for some parameter ideal $Q$ implies $R$ is regular.

\subsection{Non-negativity of $\bar{e}_1(Q)$ and $e_1^B(Q)$}
\begin{theorem}
\label{theorem: Nonnegativity}
Let $(R,\m)$ be a Noetherian local ring such that $\widehat{R}$ is reduced and equidimensional. Let $B$ be any balanced big Cohen-Macaulay $R$-algebra that satisfies (\ref{eqn: equation dagger}). Then for all parameter ideals $Q\subseteq R$ we have
$$\overline{e}_1(Q) \ge e_1^B(Q) \ge 0 \ge e_1(Q).$$
\end{theorem}

\begin{remark}
$\overline{e}_1(Q) \ge 0$ was the main theorem of \cite[Theorem 1.1]{GHM}, and $0 \ge e_1(Q)$ was first proved in full generality in \cite[Theorem 3.6]{MSV}. Our method gives alternative proofs, and is inspired by some work of Goto \cite{Go10} (in fact the proof that $e_1(Q)\leq 0$ via this method is due to Goto \cite{Go10}, see also \cite[Theorem 1.1]{HH11} for a generalization). 
\end{remark}

\begin{corollary}
\label{corollary: tightHilbertCoefficient}
Let $(R,\m)$ be an excellent local ring of characteristic $p>0$ such that $\widehat{R}$ is reduced and equidimensional. Then we have $e^*_1(Q)\geq 0$ for all parameter ideals $Q\subseteq R$.
\end{corollary}
\begin{proof}
This follows from Theorem \ref{theorem: Nonnegativity} and Remark \ref{remark: tightvsBCM}.
\end{proof}

\begin{proof}[Proof of Theorem \ref{theorem: Nonnegativity}]
Let $Q = (x_1, \ldots, x_d) \subseteq R$. Set $S = R[[y_1, \ldots, y_d]]$ and $\frak q = (y_1 - x_1, \ldots, y_d - x_d) \subseteq S$. For all $n \ge 0$ we have $y_1, \ldots, y_d$ is a system of parameters on $S/\frak q^{n+1}$, and that
$$R/Q^{n+1} = S/(\frak q^{n+1} + (y_1, \ldots, y_d)).$$
We next note that
\begin{eqnarray*}
e_0(Q) = e(Q, R) &=& \chi(x_1, \ldots, x_d; R)\\
&=& \chi(x_1, \ldots, x_d, y_1, \ldots, y_d;S)\\
&=& \chi(y_1, \ldots, y_d, y_1-x_1, \ldots, y_d - x_d;S)\\
&=& \chi(y_1, \ldots, y_d; S/\frak q)\\
&=& e(y_1, \ldots, y_d; S/\frak q),
\end{eqnarray*}
where the equalities on the second and the fourth line follow from the fact that $y_1,\dots,y_d$ and $y_1 - x_1, \ldots, y_d - x_d$ are both regular sequences on $S$. Now since $S/\frak q^{n+1}$ has a filtration by $\binom{n+d}{d}$ copies of $S/\frak q$, by the additivity formula for multiplicity (see \cite[Theorem 11.2.3]{HS}) we have
$$e(y_1, \ldots, y_d; S/\frak q^{n+1}) = \binom{n+d}{d} e(y_1, \ldots, y_d; S/\frak q).$$
Putting these together, we have
$$\binom{n+d}{d} e_0(Q) = \binom{n+d}{d} e(y_1, \ldots, y_d; S/\frak q) = e(y_1, \ldots, y_d; S/\frak q^{n+1}), $$
and
$$\ell(R/Q^{n+1}) = \ell \left(\frac{S/\frak q^{n+1}}{(y_1, \ldots, y_d)S/\frak q^{n+1}} \right).$$
Since $y_1,\dots, y_d$ is a system of parameters of $S/\frak{q}^{n+1}$, we have
$$\ell \left(\frac{S/\frak q^{n+1}}{(y_1, \ldots, y_d)S/\frak q^{n+1}} \right) \ge e(y_1, \ldots, y_d; S/\frak q^{n+1}).$$
It follows that
$$\ell(R/Q^{n+1}) \ge \binom{n+d}{d} e_0(Q),$$
and thus $e_1(Q) \le 0$ (note that this does not require any assumption on $\widehat{R}$).

It remains to show that $e_1^B(Q) \ge 0$ (since $\overline{e}_1(Q)\geq e_1^B(Q)$ always holds, see the discussion in Section 2.2). Since $e_0^B(Q)=e_0(Q)$, it is enough to show that
\begin{equation}
\label{eqn: LengthMult}
\ell(R/(Q^{n+1})^B) \le \binom{n+d}{d} e_0(Q)
\end{equation}
for any balanced big Cohen-Macaulay algebra $B$. Below we will prove a slightly stronger result. Recall that for a parameter ideal $(z_1,\dots,z_d)$ of $R$, the limit closure is defined as $(z_1,\dots,z_d)^{\lim_R}:= \bigcup_t (z_1^{t+1},\dots,z_d^{t+1}): (z_1z_2\cdots z_d)^{t}$. The limit closure does not depend on the choice of the elements $z_1,\dots, z_d$ (i.e., it only depends on the ideal $(z_1,\dots,z_d)$). This is because $(z_1,\dots,z_d)^{\lim_R}/(z_1,\dots,z_d)$ is the kernel of the natural map $R/(z_1,\dots, z_d)\to H_\m^d(R)$.
\begin{claim}
\label{claim: LimitClosurePowers}
Set $\Lambda_{n+1}= \{(\alpha_1,\dots, \alpha_d)\in \mathbb{N}^d \,\ | \,\  \alpha_i\geq 1 \text{ and } \sum_{i=1}^{d} \alpha_i = 1+n \}$ and for each $\alpha =(\alpha_1,\dots,\alpha_d) \in \Lambda_{n+1}$, set $Q(\alpha)=(x_1^{\alpha_1},\dots,x_d^{\alpha_d})$. Then we have
$$\ell\left(R/(\bigcap_{\alpha\in\Lambda_{n+1}} Q(\alpha)^{\lim_R})\right) \le \binom{n+d}{d} e_0(Q).$$
\end{claim}
\begin{proof}[Proof of Claim]Recall that we have already proved that
$$\binom{n+d}{d} e_0(Q) = e(y_1, \ldots, y_d; S/\frak q^{n+1}).$$
Moreover, we always have (for example, see \cite[Theorem 9]{MQS})
$$e(y_1, \ldots, y_d; S/\frak q^{n+1}) \geq \ell \left(\frac{S/\frak q^{n+1}}{(y_1, \ldots, y_d)^{\lim_{S/\frak q^{n+1}}}} \right).$$
Therefore it is enough to prove that
\begin{equation}
\label{eqn: limit closure}
\ell \left(\frac{S/\frak q^{n+1}}{(y_1, \ldots, y_d)^{\lim_{S/\frak q^{n+1}}}} \right) \geq \ell\left(R/(\bigcap_{\alpha\in\Lambda_{n+1}} Q(\alpha)^{\lim_R})\right).
\end{equation}
Consider $z \in S$ whose image in $S/\frak{q}^{n+1}$ is contained in $(y_1, \ldots, y_d)^{\lim_{S/\frak q^{n+1}}}$. This means there exists some $t \ge 1$ such that
\begin{eqnarray*}
(y_1 y_2\cdots y_d)^{t} z &\in & (y_1^{t+1}, \ldots, y_d^{t+1}, (y_1 - x_1, \ldots, y_d - x_d)^{n+1})\\
& \subseteq & (y_1^{t+1},  \ldots, y_d^{t+1}, (y_1 - x_1)^{\alpha_1}, \ldots, (y_d - x_d)^{\alpha_d})
\end{eqnarray*}
for each $\alpha=(\alpha_1,\dots,\alpha_d)\in\Lambda_{n+1}$. This implies
$$z \in (y_1, \ldots, y_d, (y_1 - x_1)^{\alpha_1}, \ldots, (y_d - x_d)^{\alpha_d})^{\lim_S} = (y_1, \ldots, y_d, x_1^{\alpha_1}, \ldots, x_d^{\alpha_d})^{\lim_S}.$$
But since $S = R[[y_1, \ldots, y_d]]$, it is straightforward to check that
$$(y_1, \ldots, y_d, x_1^{\alpha_1}, \ldots, x_d^{\alpha_d})^{\lim_S} = (x_1^{\alpha_1}, \ldots, x_d^{\alpha_d})^{\lim_R}S + (y_1, \ldots, y_d)S.$$
Thus if the image of $z$ is contained in $(y_1, \ldots, y_d)^{\lim_{S/\frak q^{n+1}}}$, then after modulo $(y_1, \ldots, y_d)S$, $\overline{z} \in (x_1^{\alpha_1}, \ldots, x_d^{\alpha_d})^{\lim_R}$ for each $(\alpha_1,\dots,\alpha_d)\in\Lambda_{n+1}$, i.e., $\overline{z} \in \bigcap_{\alpha\in\Lambda_{n+1}} Q(\alpha)^{\lim_R}$. It follows that the natural surjection
$$S/\frak q^{n+1} \xrightarrow{\mod(y_1, \ldots, y_d)S} R/Q^{n+1}$$
induces a surjection
$$\frac{S/\frak q^{n+1}}{(y_1, \ldots, y_d)^{\lim_{S/\frak q^{n+1}}}} \twoheadrightarrow \frac{R}{\bigcap_{\alpha\in\Lambda_{n+1}} Q(\alpha)^{\lim_R}}.$$
This clearly establishes (\ref{eqn: limit closure}) and completes the proof of claim.
\end{proof}
Finally, since $x_1,\dots, x_d$ is a regular sequence on $B$, we have $Q(\alpha)^{\lim_R}\subseteq Q(\alpha)^B$ for each $\alpha$. It follows that $\bigcap_{\alpha\in\Lambda_{n+1}} Q(\alpha)^{\lim_R} \subseteq \bigcap_{\alpha\in\Lambda_{n+1}} Q(\alpha)^B$. Now if $x \in \bigcap_{\alpha\in\Lambda_{n+1}} Q(\alpha)^B$, then we have $x \in \big(\bigcap_{\alpha\in\Lambda_{n+1}} Q(\alpha)B \big) \cap R$. But since $x_1,\dots,x_d$ is a regular sequence on $B$, it is not hard to check that $\bigcap_{\alpha\in\Lambda_{n+1}} Q(\alpha)B = Q^{n+1}B$ (see \cite[Remark 3.3]{Q22} or \cite{HRS}) and thus $x\in Q^{n+1}B \cap R = (Q^{n+1})^B$. Therefore we have $\bigcap_{\alpha\in\Lambda_{n+1}} Q(\alpha)^B = (Q^{n+1})^B$. Putting these together, we have
$$\bigcap_{\alpha\in\Lambda_{n+1}} Q(\alpha)^{\lim_R} \subseteq \bigcap_{\alpha\in\Lambda_{n+1}} Q(\alpha)^B = (Q^{n+1})^B.$$
Therefore by Claim~\ref{claim: LimitClosurePowers}, we have 
$$\ell(R/(Q^{n+1})^B) \le \binom{n+d}{d} e_0(Q)$$
as wanted.
\end{proof}

\begin{remark}
With notation as in Theorem \ref{theorem: Nonnegativity}, we do not know whether we have
$$\ell\left(\frac{S/\frak q^{n+1}}{(y_1, \ldots, y_d)^{\lim_{S/\frak q^{n+1}}}}\right) = \ell \left( \frac{R}{\bigcap_{\alpha\in\Lambda_{n+1}} Q(\alpha)^{\lim_R}}\right).$$
\end{remark}

\begin{remark}
With notation as in Claim \ref{claim: LimitClosurePowers}, fix a generating set $(x_1,\dots,x_d)$ of $Q$, one may try to define $(Q^n)^{\lim} := \bigcap_{\alpha\in\Lambda_{n}} Q(\alpha)^{\lim}$ and call this the limit closure of $Q^n$. However, it is not clear to us whether this is independent of the choice of the generators $x_1,\dots,x_d$. It is also not clear to us (even when fixing the generators $(x_1,\dots,x_d)$ of $Q$) whether $\{(Q^n)^{\lim}\}_n$ form a graded family of ideals, i.e., we do not know whether $(Q^a)^{\lim} (Q^b)^{\lim} \subseteq (Q^{a+b})^{\lim}$ for all $a,b$.
\end{remark}

\subsection{Vanishing of $\bar{e}_1(Q)$} In this subsection we prove our main result. Recall that for a finitely generated $R$-module $M$, we use the notation $\nu(M)$ to denote its minimal number of generators.

\begin{theorem}
\label{theorem: Vanishing}
Let $(R,\m)$ be a Noetherian local ring such that $\widehat{R}$ is reduced and $S_2$. If $\overline{e}_1(Q)=0$ for some parameter ideal $Q\subseteq R$, then $R$ is regular and $\nu(\m/Q)\leq 1$.
\end{theorem}
\begin{proof}
We first note that if $R$ is Cohen-Macaulay, then by \cite[Corollary 4.9]{HM}, $Q$ is integrally closed.\footnote{Using the language of \cite{HM}, $\overline{e}_1(Q)=0$ in a Cohen-Macaulay ring implies that the reduction number of the filtration $\{\overline{Q^n}\}_n$ is $0$, i.e., a minimal reduction of $\overline{Q}$ is equal to $\overline{Q}$, this is saying that $Q$ is integrally closed.} But then by the main result of \cite{Go}, $R$ is regular and $\nu(\m/Q)\leq 1$.

We may assume that $R$ is complete. We use induction on $d:=\dim(R)$. If $d\leq 2$, then $R$ is Cohen-Macaulay and we are done by the previous paragraph. Now suppose $d\geq 3$ and we have established the theorem in dimension $<d$. Let $Q=(x_1,\dots,x_d)$, $R'=R[t_1,\dots,t_d]_{\m R[t_1,\dots,t_d]}$, and $x=t_1x_1+ \cdots +t_dx_d$.
\begin{claim}
We have $R'':=\widehat{R'}/x\widehat{R'}$ is reduced, equidimensional, and $S_2$ on the punctured spectrum. Moreover, we have $\overline{e}_1(QR'')=0$.
\end{claim}
\begin{proof}
This is essentially contained in \cite[Proof of Theorem 1.1]{GHM} under the assumption that $R$ is (complete and) normal. The key ingredient is \cite[Theorem 2.1]{HU}. Since \cite{HU} does not require the normal assumption, the same proof as in \cite{GHM} works in our setup. For the ease of the reader (and also because the $S_2$ on the punctured spectrum conclusion is not stated in \cite{GHM}), we give a complete and self-contained argument here.

First of all, since $R'$ is $S_2$ and $R_0$, we know that $R'/xR'$ is $S_1$ and $R_0$ (see \cite[Lemma 10]{MS22}), so $R'/xR'$ and thus $R''$ is reduced (as $R'/xR'$ is excellent). $R''$ is clearly equidimensional since $\widehat{R'}$ is so and $x$ is a parameter in $\widehat{R'}$. To see $R''$ is $S_2$ on the punctured spectrum, it is enough to show $R'/xR'$ is $S_2$ on the punctured spectrum (as $R'/xR'$ is excellent). Now we use a similar argument as in \cite[Lemma 10]{MS22} (the idea follows from \cite{Ho73}): every non-maximal $P'\in\Spec(R'/xR')$ corresponds to a prime ideal of $R'$ that contracts to a non-maximal $P\in\Spec(R)$, thus $(R'/xR')_{P'}$ is a localization of $R_P[t_1,\dots,t_d]/(t_1x_1+\cdots + t_dx_d)$, but at least one $x_i$ is invertible in $R_P$ (say $x_1$ is invertible) so the latter is isomorphic to $R_P[t_2,\dots, t_d]$, which is $S_2$ as $R_P$ is $S_2$, thus $R'/xR'$ is $S_2$ on the punctured spectrum as wanted.

It remains to show that $\overline{e}_1(QR'')=0$. By \cite[Corollary 6.8.13]{HS}, we have a short exact sequence
$$0\to R'/\overline{Q^n} \xrightarrow{\cdot x} R'/\overline{Q^{n+1}} \to R'/(x, \overline{Q^{n+1}})\to 0.$$
Since $\overline{e}_1(Q)=0$, for $n\gg0$ we have
$$\ell(R'/\overline{Q^{n+1}})=\overline{e}_0(Q)\cdot\binom{n+d}{d} + \overline{e}_2(Q)\cdot\binom{n+d-2}{d-2} + o(n^{d-2}),$$
$$\ell(R'/\overline{Q^n})=\overline{e}_0(Q)\cdot\binom{n+d-1}{d} + \overline{e}_2(Q)\cdot\binom{n+d-3}{d-2} + o(n^{d-2}).$$
It follows that
\begin{equation}\label{eqn: LengthIntClosure}
\ell(R'/(x, \overline{Q^{n+1}})) = \overline{e}_0(Q)\cdot\binom{n+d-1}{d-1} + o(n^{d-2}).
\end{equation}
We next show that for all $n\gg0$, $\overline{Q^n}(R'/xR')= \overline{Q^n(R'/xR')}$. Once this is proved, we will have $\overline{Q^n}R''= \overline{Q^nR''}$ for all $n\gg0$ by \cite[Lemma 9.1.1]{HS} and thus (\ref{eqn: LengthIntClosure}) will tell us that $$\ell(R''/\overline{Q^{n+1}R''})= \overline{e}_0(Q)\binom{n+d-1}{d-1}+o(n^{d-2}).$$
Since $x$ is a general element of $Q$, we have $\overline{e}_0(Q)=e(Q, R')=e(QR'', R'')=\overline{e}_0(QR'')$ and so the above equation implies that $\overline{e}_1(QR'')=0$ as wanted.

To show $\overline{Q^n}(R'/xR')= \overline{Q^n(R'/xR')}$ for $n\gg0$, let $\mathcal{R}'$ denote the integral closure of $R'[Qt, t^{-1}]$ inside $R'[t, t^{-1}]$. Concretely, $\mathcal{R}'$ is the $\mathbb{Z}$-graded ring such that $\mathcal{R}_n= \overline{Q^n}t^n$ for $n>0$ and $\mathcal{R}'_n=R't^{n}$ for $n\leq 0$. Consider the map
\begin{equation*}
\mathcal{R}'/(xt)\mathcal{R}' \to R'[t, t^{-1}]/(xt)R'[t, t^{-1}].
\end{equation*}
If we localize at any prime ideal $\mathcal{P}$ of $R'[Qt, t^{-1}]$ that does not contain $(Qt, t^{-1})$, then we note that $(\mathcal{R}'/(xt)\mathcal{R}')_{\mathcal{P}}$ is integrally closed inside $(R'[t, t^{-1}]/(xt)R'[t, t^{-1}])_{\mathcal{P}}$. To see this, one can ``unlocalize" the ring $R'$, and consider the integral closure of $R[t_1,\dots, t_d][Qt, t^{-1}]$ inside $R[t_1,\dots,t_d][t, t^{-1}]$, call this ring $\mathcal{R}$. If one localizes the map $\mathcal{R}/(xt)\mathcal{R}\to R[t_1,\dots,t_d][t, t^{-1}]/(xt)R[t_1,\dots,t_d][t, t^{-1}]$ at any prime ideal that does not contain $(Qt, t^{-1})$ (say it does not contain $x_1t$), then the resulting map is a localization of $R[t_2,\dots,t_d][\overline{Q^n}t^n, t^{-1}][\frac{1}{x_1t}] \to R[t_2,\dots, t_d][t, t^{-1}][\frac{1}{x_1t}]$, and the former is already integrally closed in the latter. 

Since the radical of $(Qt, t^{-1})$ is the unique homogeneous maximal ideal of $R[Qt, t^{-1}]$, it follows that $\mathcal{R}'/(xt)\mathcal{R}'$ and the integral closure of $R'[Qt, t^{-1}]/(xt)R'[Qt, t^{-1}]$ inside $R'[t, t^{-1}]/(xt)R'[t, t^{-1}]$ agree in large degree. But note that for $n>0$,
$$[\mathcal{R}'/(xt)\mathcal{R}']_n\cong \frac{\overline{Q^n}}{x\overline{Q^{n-1}}}\cdot t^n\cong \frac{\overline{Q^n}}{x(\overline{Q^n}: x)}\cdot t^n\cong \frac{\overline{Q^n}}{(xR')\cap \overline{Q^n}}\cdot t^n\cong  \overline{Q^n}(R'/xR')\cdot t^n,$$
where we have used \cite[Corollary 6.8.13]{HS} again, while the degree $n$ part of the integral closure of $R'[Qt, t^{-1}]/(xt)R'[Qt, t^{-1}]$ inside $R'[t, t^{-1}]/(xt)R'[t, t^{-1}]$ is $\overline{Q^n(R'/xR')}\cdot t^n$. Thus the fact that they agree in degree $n\gg0$ is precisely saying that $\overline{Q^n}(R'/xR')= \overline{Q^n(R'/xR')}$ for $n\gg0$.
\end{proof}

Now we come back to the proof of the theorem. Let $S$ be the $S_2$-ification of $R''$. We have a short exact sequence
$$0\to R'' \to S \to S/R'' \to 0$$
such that $S/R''$ has finite length (since $R''$ is $S_2$ on the punctured spectrum). Also note that $(S,\n)$ is (complete) local by \cite[Proposition (3.9)]{HH94} and that $S$ is reduced (since $S$ is a subring of the total quotient ring of $R''$). Since $R''\to S$ is an integral extension, we have $\overline{IS}\cap R'' =\overline{I}$ for every ideal $I\subseteq R''$ by \cite[Proposition 1.6.1]{HS}. It follows that
$\ell_{R''}(S/\overline{Q^nS})\geq \ell_{R''}(R''/\overline{Q^nR''})$ for all $n\geq 0$. Thus for $n\gg0$ we have
\begin{eqnarray*}
    &&  \overline{e}_0(QR'')\binom{n+d}{d} -\overline{e}_1(QR'')\binom{n+d-1}{d-1} + o(n^{d-1})\\
   &=& \ell_{R''}(R''/\overline{Q^{n+1}R''}) \\
   &\leq&  \ell_{R''}(S/\overline{Q^{n+1}S}) \\
   &=& [S/\n: R/\m]\cdot \ell_S(S/\overline{Q^{n+1}S})\\
   &=& [S/\n: R/\m]\cdot \left(\overline{e}_0(QS)\binom{n+d}{d} -\overline{e}_1(QS)\binom{n+d-1}{d-1}+o(n^{d-1})\right).
\end{eqnarray*}
Since $S$ is a rank one module over $R''$, we also know that
$$\overline{e}_0(QR'')=e(QR'', R'') = [S/\n: R/\m] \cdot e(QS, S)=[S/\n: R/\m] \cdot \overline{e}_0(QS),$$
where the second equality is the projection formula for Hilbert-Samuel multiplicity (which can be seen by combining \cite[Theorem 11.2.4 and Theorem 11.2.7]{HS}).
Putting these together we have
$$ [S/\n: R/\m] \cdot \overline{e}_1(QS)\leq \overline{e}_1(QR'')=0.$$
But since $\overline{e}_1(QS)\geq 0$ by \cite[Theorem 1.1]{GHM} (see Theorem \ref{theorem: Nonnegativity}), we must have $\overline{e}_1(QS)=0$. Now $(S,\n)$ is a reduced complete local ring that is $S_2$ and $\dim(S)=d-1$, such that $\overline{e}_1(QS)=0$. By our inductive hypothesis, we know that $S$ is regular. But since $S/R''$ has finite length, by the long exact sequence of local cohomology induced by
$0\to R'' \to S \to S/R'' \to 0$, we obtain that
$$H_\m^i(R'')=0 \text{ for all $i<\dim(R'')$ and $i\neq 1$, and } H_\m^1(R'')\cong S/R''.$$
At this point, we consider the long exact sequence of local cohomology induced by $0\to \widehat{R'} \xrightarrow{\cdot x} \widehat{R'} \to R''\to 0$, we get
$$0=H_\m^1(R')\to H_\m^1(R'') \to H_\m^2(R')\xrightarrow{\cdot x} H_\m^2(R') \to H_\m^2(R'') \to \cdots.$$

If $d\geq 4$, then $\dim(R'')\geq 3$ and thus $H_\m^2(R'')=0$. Since $\widehat{R'}$ is $S_2$, $H_\m^2(R')$ has finite length and the above exact sequence tells us that $H_\m^2(R')=0$ by Nakayama's lemma. But then by the above exact sequence again, we have $H_\m^1(R'')=0$ and hence $S/R''=0$. Thus $R''\cong S$ is regular. But then $\widehat{R'}$ and hence $R$ is regular as wanted.

Finally, suppose $d=3$. Let $B$ be a balanced big Cohen-Macaulay algebra of $\widehat{R'}$ that is $\m$-adic complete, then $B/xB$ is a balanced big Cohen-Macaulay algebra of $R''$. It follows that the canonical map $R''\to B/xB$ factors through $S$.
\begin{claim}
$B/xB$ is a balanced big Cohen-Macaulay algebra over $S$.
\end{claim}
\begin{proof}[Proof of Claim]
It is clear that some system of parameters of $S$ (namely those coming from $R''$) are regular sequences on $B/xB$. To see that every system of parameters of $S$ is a regular sequence on $B/xB$, we first note that $B/xB$ is $\m$-adically complete: since $B$ is $\m$-adic complete, $B/xB$ is derived $\m$-complete by \cite[Tag 091U]{Sta}, take $(y,z)$ that is a system of parameters of $R''$, then as $y,z$ is a regular sequence on $B/xB$, the derived completion with respect to $(y,z)$, which is $B/xB$ itself, agrees with the usual completion with respect to $(y,z)$ by \cite[Tag 0920]{Sta} (equivalently, with respect to $\m$ as $\sqrt{(y,z)}=\m$). Hence by \cite[Corollary 8.5.3]{BH}, every system of parameters of $S$ is a regular sequence on $\widehat{B/xB}\cong B/xB$.
\end{proof}   
Note that $\dim(R'')=\dim(S)=2$ and $S$ is regular, thus the long exact sequence of local cohomology induced by
$0\to R'' \to S \to S/R'' \to 0$ implies that $H_\m^2(R'')\cong H_\m^2(S)$. Hence we have the following commutative algebra:
\[\xymatrix{
&& H_\m^2(S) \ar@{=}[d]& &&\\
H_\m^2(R') \ar[r]^{\cdot x} & H_\m^2(R') \ar[r] \ar[d] & H_\m^2(R'') \ar[r] \ar@{^{(}->}[d]^\phi & H_\m^3(R') \ar[d] \ar[r] & H_\m^3(R')\ar[d] \ar[r] & 0\\
& 0= H_\m^2(B) \ar[r] & H_\m^2(B/xB) \ar[r] & H_\m^3(B) \ar[r]  & H_\m^3(B)  \ar[r] & 0
}
\]
where the injectivity of $\phi$ follows from the fact that $B/xB$ is a balanced big Cohen-Macaulay algebra over $S$ and thus faithfully flat over $S$ (as $S$ is regular). Chasing this diagram we find that the map $H_\m^2(R')\xrightarrow{\cdot x} H_\m^2(R')$ is surjective. But since $\widehat{R'}$ is $S_2$, $H_\m^2(R')$ has finite length, thus $H_\m^2(R')=0$ by Nakayama's lemma. Hence $\widehat{R'}$ is Cohen-Macaulay and thus $R''$ is also Cohen-Macaulay. But then $R''\cong S$ and so $R''$ is regular and thus $\widehat{R'}$ is regular. Thus $R$ is regular as wanted.

Now we have established that $R$ is regular, we can repeat the argument in the first paragraph of the proof to show that $\nu(\m/Q)\leq 1$ (essentially, this follows from the main result of \cite{Go}).
\end{proof}

As a consequence, we answer the problem raised in \cite[Section 3]{GHM} for excellent rings.

\begin{corollary}
Let $R$ be an excellent local ring such that $\widehat{R}$ is reduced and equidimensional. Suppose $I\subseteq R$ is an $\m$-primary ideal such that $\overline{e}_1(I)=0$. Then $R^{\emph{N}}$, the normalization of $R$, is regular and $IR^\emph{N}$ is normal (i.e., all powers of $IR^\emph{N}$ are integrally closed in $R^\emph{N}$).
\end{corollary}
\begin{proof}
Replacing $R$ by $R[t]_{\m R[t]}$, we may assume that the residue field of $R$ is infinite (we leave it to the readers to check that the hypotheses and conclusions are stable under such a base change). Let $S$ be the $S_2$-ification of $R$. We will show that the $\m$-adic completion of $\widehat{S}$ is regular. Since $R$ is excellent, $\widehat{S}$ agrees with the $S_2$-ification of $\widehat{R}$ by \cite[Proposition 3.8]{HH94}. Thus $\widehat{S}$ is semilocal, reduced, and $S_2$. Since $\overline{J\widehat{S}}\cap \widehat{R}=\overline{J}$ for every $\m$-primary ideal $J\subseteq \widehat{R}$ by \cite[Proposition 1.6.1]{HS}, we have $\ell_{\widehat{R}}(\widehat{R}/\overline{J})\leq \ell_{\widehat{R}}(\widehat{S}/\overline{J\widehat{S}})$.

Let $\n_1,\dots, \n_s$ be the maximal ideals of $\widehat{S}$ and let $S_i:= (\widehat{S})_{\n_i}$ (in fact, since $\widehat{S}$ is complete, we have $\widehat{S}\cong \prod_{i=1}^{s}S_i$, and each $S_i$ is complete local, reduced, and $S_2$). Then we have 

\begin{eqnarray*}
    &&  \overline{e}_0(I)\binom{n+d}{d} -\overline{e}_1(I)\binom{n+d-1}{d-1} + o(n^{d-1})\\
   &=& \ell_R(R/\overline{I^{n+1}})=\ell_{\widehat{R}}(\widehat{R}/\overline{I^{n+1}\widehat{R}}) \\
   &\leq&  \ell_{\widehat{R}}(\widehat{S}/\overline{I^{n+1}\widehat{S}}) \\
   &=& \sum_{i=1}^{s}[S_i/\n_i: R/\m]\cdot \ell_{S_i}(S_i/\overline{I^{n+1}S_i})\\
   &=& \sum_{i=1}^{s}[S_i/\n_i: R/\m]\cdot \left(\overline{e}_0(IS_i)\binom{n+d}{d} -\overline{e}_1(IS_i)\binom{n+d-1}{d-1}+o(n^{d-1})\right),
\end{eqnarray*}
where we have used \cite[Lemma 9.1.1]{HS} for the equality in the second line. Since $\widehat{S}$ is a rank one module over $\widehat{R}$, we also know that
$$\overline{e}_0(I)=e(I\widehat{R}, \widehat{R}) = \sum_{i=1}^{s}[S_i/\n_i: R/\m] \cdot e(IS_i, S_i)=\sum_{i=1}^{s}[S_i/\n_i: R/\m] \cdot\overline{e}_0(IS_i),$$
where we have used the projection formula for the Hilbert-Samuel multiplicity (see \cite[Theorem 11.2.4 and Theorem 11.2.7]{HS}).
The above inequality implies that
$$\sum_{i=1}^{s}[S_i/\n_i: R/\m] \cdot\overline{e}_1(IS_i) \leq \overline{e}_1(I)=0.$$
But since $\overline{e}_1(IS_i)\geq 0$ by \cite[Theorem 1.1]{GHM}, we must have $\overline{e}_1(IS_i)=0$ for all $i$. Let $Q$ be a minimal reduction of $I$ (note that $Q$ is a parameter ideal of $R$, since we have reduced to the case that $R$ has an infinite residue field). It follows that $\overline{e}_1(QS_i)=0$ and thus by Theorem \ref{theorem: Vanishing}, $S_i$ is regular and $\nu(\n_i/Q)\leq 1$. But then $QS_i$ is normal in $S_i$. It follows that $\widehat{S}\cong \prod_{i=1}^{s}S_i$ is regular,  $Q\widehat{S}$ is normal in $\widehat{S}$ and in particular, $Q\widehat{S}=I\widehat{S}$.

Since $S\to \widehat{S}\cong \widehat{R}\otimes_RS$ is faithfully flat with geometrically regular fibers (as $R$ is excellent). We have $S$ is regular and $QS=IS$ is normal in $S$ by \cite[Theorem 19.2.1]{HS}. Finally, since $S$ is regular, $S$ agrees with the normalization $R^\text{N}$ of $R$.
\end{proof}

\begin{remark}
\label{remark: S_2necessary}
The condition $\widehat{R}$ is $S_2$ cannot be dropped in Theorem \ref{theorem: Vanishing}. This was already observed in \cite[Section 3]{GHM}. We give a different example that is a complete local domain. Let $R=k[[x,xy,y^2, y^3]]$ where $k$ is a field. Then the $S_2$-ification of $R$ is $S=k[[x,y]]$ and we have $0\to R\to S\to S/R\cong k\cdot \overline{y}\to 0$. Let $Q=(x, y^2)\subseteq R$ and we claim that $\overline{e}_1(Q)=0$. To see this, note that $QS=(x,y^2)\subseteq S$ is normal and $\ell(S/Q^{n+1}S)=2\cdot \binom{n+2}{2}$. It follows from the short exact sequence
$$0\to R/\overline{Q^{n+1}}\to S/Q^{n+1}S \to k\to 0$$
that $\ell(R/\overline{Q^{n+1}})=2\cdot \binom{n+2}{2} -1$. In particular, $\overline{e}_1(Q)=0$.
\end{remark}

Recall that a Noetherian local ring $(R,\m)$ of prime characteristic $p>0$ is called $F$-rational if every ideal generated by a system of parameters is tightly closed. It was mentioned in \cite{DQV} that Huneke asked that when $\widehat{R}$ is reduced and equidimensional of prime characteristic $p>0$, whether $e_1^*(Q)=0$ for some system of parameters $Q\subseteq R$ implies $R$ is $F$-rational. In general, counter-examples to the question were constructed in \cite[Example 5.4 and 5.5]{DQV} (in fact, the example in Remark \ref{remark: S_2necessary} is a counter-example that is a complete local domain). However, all these examples do not satisfy Serre's $S_2$ condition. 

Let $(R,\m)$ be a Noetherian local ring and let $B$ be a big Cohen-Macaulay $R$-algebra. Recall that $R$ is called ${\rm{BCM}}_B$-rational if $R$ is Cohen-Macaulay and the natural map $H_\m^d(R)\to H_\m^d(B)$ is injective, where $d=\dim(R)$. If $R$ is an excellent local ring of prime characteristic $p>0$, then $R$ is $F$-rational if and only if $R$ is ${\rm{BCM}}_B$-rational for all big Cohen-Macaulay algebra $B$, see \cite[Proposition 3.5]{MS}.

We propose the following conjecture relating the vanishing of $e_1^B(Q)$ and ${\rm{BCM}}_B$-rational singularities, which modifies Huneke's question and makes sense in all characteristics.

\begin{conjecture}
\label{conjecture: BCMrational}
Let $(R,\m)$ be a Noetherian local ring such that $\widehat{R}$ is reduced and $S_2$. Let $B$ be a balanced big Cohen-Macaulay $R$-algebra that satisfies (\ref{eqn: equation dagger}). If $e_1^B(Q) = 0$ for some parameter ideal $Q\subseteq R$, then $R$ is ${\rm{BCM}}_B$-rational.

In particular, if $R$ is excellent and has characteristic $p>0$ (such that $\widehat{R}$ is reduced and $S_2$), and $e_1^*(Q) = 0$ for some parameter ideal $Q\subseteq R$, then $R$ is F-rational.
\end{conjecture}

We have the following partial result towards the Conjecture \ref{conjecture: BCMrational}, which is an analog of the main result of \cite{MTV}.

\begin{proposition}
\label{proposition: BCMrational}
Let $(R, \frak m)$ be a Noetherian local ring such that $\widehat{R}$ is reduced and equidimensional. Let $B$ be a balanced big Cohen-Macaulay $R$-algebra that satisfies (\ref{eqn: equation dagger}). If $e_1^B(Q) = e_1(Q)$ for some parameter ideal $Q\subseteq R$, then $R$ is ${\rm{BCM}}_B$-rational.

In particular, if $R$ is excellent and has characteristic $p>0$, and $e_1^*(Q) = e_1(Q)$ for some parameter ideal $Q\subseteq R$, then $R$ is F-rational.
\end{proposition}
\begin{proof}
By Theorem \ref{theorem: Nonnegativity}, we know that $e_1^B(Q)=e_1(Q)=0$. By the main result of \cite{GGH10}, $e_1(Q)=0$ implies that $R$ is Cohen-Macaulay.
By \cite[Corollary 4.9]{HM}, we have $Q^B = Q$. Now we consider the commutative diagram:
\[
\xymatrix{
R/Q \ar@{^{(}->}[d]  \ar@{^{(}->}[r]      & B/QB \ar@{^{(}->}[d]\\
H^d_{\frak m}(R) \ar[r]   & H^d_{\frak m}(B)
          }
\]
where the injectivity of the top row follows from $Q^B=Q$, the injectivity of the left column is because $R$ is Cohen-Macaulay, and the injectivity of the right column is because $B$ is balanced big Cohen-Macaulay. Since $R$ is Cohen-Macaulay, we know that $\mathrm{Soc}(R/Q) \cong \mathrm{Soc}(H^d_{\frak m}(R))$. Chasing the commutative diagram we find that $H^d_{\frak m}(R) \to H^d_{\frak m}(B)$ injective. Therefore $R$ is $\mathrm{BCM}_B$-rational.
\end{proof}

\begin{remark}
It is clear from the proof of Proposition \ref{proposition: BCMrational} that Conjecture \ref{conjecture: BCMrational} holds when $R$ is Cohen-Macaulay, and this essentially follows from \cite[Corollary 4.9]{HM}.
\end{remark}


\begin{thebibliography}{17}



\bibitem{A18} Yves Andr\'{e}, \emph{La conjecture du facteur direct}, Publ. Math. Inst. Hautes \'{E}tudes Sci. {\bf 127} (2018), 71--93.

\bibitem{A20} Yves Andr\'{e}, \emph{Weak functoriality of Cohen-Macaulay algebras}, J. Amer. Math. Soc. {\bf 33} (2020), no. 2, 363--380.

\bibitem{B20} Bhargav Bhatt, \emph{Cohen-Macaulayness of absolute integral closures}, \url{https://arxiv.org/abs/2008.08070}.


\bibitem{BH} Winfried Bruns and J\"{u}rgen Herzog, \emph{Cohen-Macaulay rings}, Cambridge Studies in Advanced Mathematics, {\bf 39}. Cambridge University Press, Cambridge, 1993. xii+403 pp.

\bibitem{DQV} Saipriya Dubey, Pham H. Quy and Jugal Verma, \emph{Tight Hilbert Polynomial and F-rational local rings}, Res. Math. Sci. {\bf 10} (2023), 8.

\bibitem{GGH10} L. Ghezzi, S. Goto, J. Hong, K. Ozeki, T. T. Phuong, and W. V. Vasconcelos \emph{Cohen–Macaulayness versus the vanishing of the first Hilbert coefficient of parameter ideals}, J. London. Math. Soc. {\bf 81} (2010), 679--695.

\bibitem{GMV} Kriti Goel, Jugal K. Verma and Vivek Mukundan, \emph{Tight closure of powers of ideals and tight Hilbert polynomials}, Math. Proc. Cambridge Philos. Soc. {\bf 169} (2020), no. 2, 335--355.

\bibitem{Go} Shiro Goto, \emph{Integral closedness of complete-intersection ideals}, J. Algebra {\bf 108} (1987), 151--160.

\bibitem{Go10} Shiro Goto, \emph{Hilbert cofficients of parameters}, Proceeding of the $5^{\rm th}$ Japan-Vietnam joint seminar on commutative algebra 2010, 1--49.

\bibitem{GHM} Shiro Goto, Jooyoun Hong and Mousumi Mandal, \emph{The positivity of the first coefficients of normal Hilbert polynomials}, Proc. Amer. Math. Soc. {\bf 139} (2011), no. 7, 2399--2406.


\bibitem{HRS} William Heinzer, Louis J. Ratliff and Kishor Shah, \emph{Parametric decomposition of monomial ideals. I}, Houston J. Math. {\bf 21} (1995), no. 1, 29--52.

\bibitem{HM18} Raymond Heitmann and Linquan Ma, \emph{Big Cohen-Macaulay algebras and the vanishing conjecture for maps of Tor in mixed characteristic}, Algebra Number Theory {\bf 12} (2018), no. 7, 1659--1674.

\bibitem{HH11} Futoshi Hayasaka and Eero Hyry, \emph{On the Buchsbaum–Rim function of a parameter module}, J. Algebra {\bf 327} (2011), no. 1, 307--315.

\bibitem{Ho73} Melvin Hochster, \emph{Properties of Noetherian rings stable under general grade reduction}, Arch. Math. (Basel) {\bf 24} (1973), 393--396.

\bibitem{Ho07} Melvin Hochster, \emph{Foundations of tight closure theory}, unpublished lecture notes, \url{http://www.math.lsa.umich.edu/~hochster/711F07/fndtc.pdf}.

\bibitem{Ho94} Melvin Hochster, \emph{Solid closure}, Commutative algebra: syzygies, multiplicities, and birational algebra (South Hadley, MA, 1992), 103--172, Contemp. Math., {\bf 159}, Amer. Math. Soc., Providence, RI, 1994.


\bibitem{HH94} Melvin Hochster and Craig Huneke, \emph{Indecomposable canonical modules and connectedness}, Commutative algebra: syzygies, multiplicities, and birational algebra (South Hadley, MA, 1992), 197--208, Contemp. Math., {\bf 159}, Amer. Math. Soc., Providence, RI, 1994.

\bibitem{HH95} Melvin Hochster and Craig Huneke, \emph{Applications of the existence of big Cohen-Macaulay algebras}, Adv. Math. {\bf 113} (1995), no. 1, 45--117.

\bibitem{HM} Sam Huckaba and Thomas Marley, \emph{Hilbert coefficients and the depths of associated graded
rings} J. London Math. Soc., {\bf 56} (1997), no. 1, 64--76.

\bibitem{H98} Craig Huneke, \emph{Tight closure, parameter ideals, and geometry.} In Six lectures on commutative algebra (Bellaterra, 1996), volume {\bf 166} of Progr. Math., pages 187--239. Birkh\"{a}user, Basel, 1998.

\bibitem{HS} Craig Huneke and Irena Swanson, \emph{Integral closure of ideals, rings, and modules}, London Mathematical Society Lecture Note Series, vol. 336, Cambridge, UK: Cambridge University Press, 2006

\bibitem{HU} Jooyoun Hong and Bernd Ulrich, \emph{Specialization and integral closure}, J. Lond. Math. Soc. (2) {\bf 90} (2014), no. 3, 861--878.


\bibitem{MQS} Linquan Ma, Pham H. Quy and Ilya Smirnov, \emph{Colength, multiplicity, and ideal closure operations}, Comm. Algebra, {\bf 48} (2020), no. 4, 1601--1607.

\bibitem{MS} Linquan Ma and Karl Schwede, \emph{Singularities in mixed characteristic via perfectoid big Cohen-Macaulay algebras}, Duke Math. J. {\bf 170} (2019), no. 13, 2815--2890.

\bibitem{MS22} Linquan Ma and Ilya Smirnov, \emph{Uniform Lech's inequality}, Proc. Amer. Math. Soc. {\bf 151} (2023), 2387--2397.



\bibitem{MSV}Mousumi Mandal,  Balwant Singh and Jugal K. Verma, \emph{On some conjectures about the Chern numbers of filtrations}, J. Algebra {\bf 325} (2011), 147--162.

\bibitem{MTV} Marcel Moral\`{e}s, Ng\^{o} V. Trung and Orlando Villamayor, \emph{Sur la fonction de {H}ilbert-{S}amuel des cl\^{o}tures int\'{e}grales	des puissances d'id\'{e}aux engendr\'{e}s par un syst\`eme de param\`etres}, J. Algebra, {\bf 129} (1990), no. 1, 96--102.

\bibitem{Q22} Pham H. Quy, \emph{Uniform annihilators of systems of parameters}, Proc. Amer. Math. Soc. {\bf 150} (2022), no. 5, 1937--1947.

\bibitem{Sta} The Stacks project, \url{https://stacks.math.columbia.edu}.

\bibitem{V} Wolmer V. Vasconcelos, \emph{The Chern coefficients of local rings}, Michigan Math. J. 57 (2008), 725--743.


\end{thebibliography}
\end{document}